\documentclass[11pt]{amsart}
\usepackage{amsmath}
\usepackage{amscd}
\usepackage{amssymb}
\usepackage{amsthm}
\usepackage{dsfont} 
\usepackage{appendix} 
\usepackage{verbatim}
\usepackage[normalem]{ulem} 
\usepackage[all]{xy} 

\begin{document}


\theoremstyle{plain}
\newtheorem{theorem}{Theorem}[section]
\newtheorem{introtheorem}{Theorem}

\theoremstyle{plain}
\newtheorem{proposition}[theorem]{Proposition}
\newtheorem{prop}[theorem]{Proposition}

\theoremstyle{plain}
\newtheorem{corollary}[theorem]{Corollary}

\theoremstyle{plain}
\newtheorem{lemma}[theorem]{Lemma}

\theoremstyle{plain}
\newtheorem{expectation}[theorem]{Expectation}

\theoremstyle{remark}
\newtheorem{remind}[theorem]{Reminder}

\theoremstyle{definition}
\newtheorem{condition}[theorem]{Condition}

\theoremstyle{definition}
\newtheorem{construction}[theorem]{Construction}

\theoremstyle{definition}
\newtheorem{definition}[theorem]{Definition}

\theoremstyle{definition}
\newtheorem{question}[theorem]{Question}

\theoremstyle{definition}
\newtheorem{example}[theorem]{Example}

\theoremstyle{definition}
\newtheorem{notation}[theorem]{Notation}

\theoremstyle{definition}
\newtheorem{convention}[theorem]{Convention}

\theoremstyle{definition}
\newtheorem{assumption}[theorem]{Assumption}
\newtheorem{isoassumption}[theorem]{Isomorphism Assumption}

\newtheorem{finassumption}[theorem]{Finiteness Assumption}
\newtheorem{indhypothesis}[theorem]{Inductive Hypothesis}

\theoremstyle{remark}
\newtheorem{remark}[theorem]{Remark}

\numberwithin{equation}{subsection}





\newcommand{\mathscr}[1]{\mathcal{#1}}
\newcommand{\TODO}{{\color{red} TODO}}
\newcommand{\CHANGE}{{\color{red} CHANGE}}
\newcommand{\annette}[1]{{\color{red}Annette: #1 }}
\newcommand{\simon}[1]{{\color{green}Simon: #1 }}
\newcommand{\giuseppe}[1]{{\color{blue}Giuseppe: #1 }}

\newcommand{\tensor}{\otimes}
\newcommand{\Fil}{\tn{Fil}}
\newcommand{\Gh}{\mathcal{G}}
\newcommand{\Fh}{\shfF}
\newcommand{\Oh}{\mathcal{O}}
\newcommand{\GrVec}{\textnormal{GrVec}}
\newcommand{\GSp}{\textnormal{GSp}}
\newcommand{\BN}{\textnormal{BN}}
\newcommand{\modulo}{\textnormal{mod}}
\newcommand{\CH}{\mathrm{CH}}
\newcommand{\im}{\mathrm{im}}
\newcommand{\cl}{\mathrm{cl}}
\newcommand{\prim}{\mathrm{prim}}
\newcommand{\Clo}{\mathrm{Clo}}
\newcommand{\CHM}{\mathrm{CHM}}
\newcommand{\GL}{\mathrm{GL}}
\newcommand{\NUM}{\mathrm{NUM}}
\newcommand{\ab}{\mathrm{ab}}
\newcommand{\num}{\mathrm{num}}
\newcommand{\id}{\mathrm{id}}
\newcommand{\isocan}{\xrightarrow{\hspace{1.85pt}\sim \hspace{1.85pt}}}
\newcommand{\isom}{\cong}
\newcommand{\red}{\mathrm{red}}
\newcommand{\Betti}{R_B}
\newcommand{\ladic}{R_\ell}
\newcommand{\Hodge}{R_H}
\newcommand{\MHM}{\mathrm{MHM}}
\newcommand{\Hh}{\mathcal{H}}
\newcommand{\A}{\mathbb{A}}
\newcommand{\gm}{\mathrm{gm}}
\newcommand{\qfh}{\mathrm{qfh}}
\newcommand{\DM}{\tn{\tbf{DM}}}
\newcommand{\DA}{\tn{\tbf{DA}}}
\newcommand{\DMeff}{\tn{\tbf{DM}}^{\eff}}
\newcommand{\DAeff}{\tn{\tbf{DA}}^{\eff}}
\newcommand{\DAgm}{\tn{\tbf{DA}}_{\gm}}
\newcommand{\Bei}{\mathcyr{B}}
\newcommand{\kd}{\tn{kd}}
\newcommand{\Sm}{\tn{\tbf{Sm}}}
\newcommand{\SmCor}{\tn{\tbf{SmCor}}}
\newcommand{\cor}{\tn{\tbf{cor}}}
\newcommand{\Mor}{\textnormal{Mor}}
\newcommand{\Hom}{\textnormal{Hom}}
\newcommand{\End}{\textnormal{End}}
\newcommand{\sss}{\textnormal{ss}}
\newcommand{\Sh}{\tn{\tbf{Sh}}}
\newcommand{\et}{\tn{\'{e}t}}
\newcommand{\an}{\tn{an}}
\newcommand{\D}{\tn{D}}
\newcommand{\eff}{\tn{eff}}
\newcommand{\DMgm}{\DM_{\tn{gm}}}
\newcommand{\vp}{\varphi}
\newcommand{\Sym}{\textnormal{Sym}}
\newcommand{\OSym}{\textnormal{coSym}}
\newcommand{\Mof}{M_1}
\newcommand{\CGS}{\tn{\tbf{cGrp}}}
\newcommand{\tr}{\textnormal{tr}}
\newcommand{\one}{\mathds{1}}
\newcommand{\Spec}{\textnormal{Spec}}
\newcommand{\Sch}{\tn{\tbf{Sch}}}
\newcommand{\Nor}{\tn{\tbf{Nor}}}
\newcommand{\Lie}{\tn{Lie}}
\newcommand{\Ker}{\tn{Ker}}
\newcommand{\Frob}{\tn{Frob}}
\newcommand{\Ver}{\tn{Ver}}
\newcommand{\N}{\mathbb{N}}
\newcommand{\Z}{\mathbb{Z}}
\newcommand{\Q}{\mathbb{Q}}
\newcommand{\Ql}{\mathbb{Q}_{\ell}}
\newcommand{\R}{\mathbb{R}}
\newcommand{\C}{\mathbb{C}}
\newcommand{\G}{\mathbb{G}}

\newcommand{\xra}{\xrightarrow}
\newcommand{\xla}{\xleftarrow}
\newcommand{\sxra}[1]{\xra{#1}}
\newcommand{\sxla}[1]{\xla{#1}}

\newcommand{\sra}[1]{\stackrel{#1}{\ra}}
\newcommand{\sla}[1]{\stackrel{#1}{\la}}
\newcommand{\slra}[1]{\stackrel{#1}{\lra}}
\newcommand{\sllra}[1]{\stackrel{#1}{\llra}}
\newcommand{\slla}[1]{\stackrel{#1}{\lla}}

\newcommand{\ira}{\stackrel{\simeq}{\ra}}
\newcommand{\ila}{\stackrel{\simeq}{\la}}
\newcommand{\ilra}{\stackrel{\simeq}{\lra}}
\newcommand{\illa}{\stackrel{\simeq}{\lla}}

\newcommand{\ul}[1]{\underline{#1}}
\newcommand{\tn}[1]{\textnormal{#1}}
\newcommand{\tbf}[1]{\textbf{#1}}

\newcommand{\Spt}{\mathop{\mathbf{Spt}}\nolimits}
\newcommand{\ra}{\rightarrow}
\newcommand{\old}{\mathrm{old}}

\setcounter{tocdepth}{1}

\title{Conservativity of realizations on motives of abelian type over $\mathbb{F}_q$}

\author{Giuseppe Ancona}
\address{Institut de Recherche Math\'ematique Avanc\'ee, Universit\'e de Strasbourg}
\email{ancona@math.unistra.fr}

\begin{abstract}
We show that the $\ell$-adic realization functor is conservative when restricted to the Chow motives of abelian type over a finite field.

A weak version of this conservativity result extends to mixed motives of abelian type.
\end{abstract}

\maketitle
\begin{center}
\today
\end{center}
\tableofcontents
\section*{Introduction}

Let $k$ be a base field and $\DMgm(k)_{\Q}$ be Voevodsky's category of mixed motives over $k$ with rational coefficients. Let $\ell$ be a prime number invertible in $k$, and consider the $\ell$-adic realization functor
\[R_{\ell}: \DMgm(k)_{\Q} \rightarrow D^b(\Q_{\ell})\]
to the bounded derived category of $\Q_{\ell}$-vector spaces.

One of the central conjectures in motives predicts that this functor is conservative (i.e. it detects isomorphisms), see  \cite{consev}  for an overview on this conjecture.  This conjecture is deep and still widely open: the case of surfaces would imply Bloch's conjecture for surfaces.

In this paper we focus on the dimension one case (ore equivalently on abelian varieties), more precisely we deal with the following categories.
\begin{definition}\label{defabeliancat}
Define $\CHM^{\ab}(k)_{\Q}$ to be the smallest rigid and pseudo-abelian full subcategory of  $\DMgm(k)_{\Q}$ containing motives of abelian varieties. Define $\DM^{\ab}(k)_{\Q}\supset \CHM^{\ab}(k)_{\Q}$ to be the smallest triangulated, rigid and pseudo-abelian full subcategory of  $\DMgm(k)_{\Q}$  containing motives of abelian varieties. 
\end{definition}

In characteristic zero Wildeshaus showed that $R_{\ell}$ is conservative when restricted to $\DM^{\ab}(k)_{\Q}$ \cite[Theorem 1.12]{WildPic}. He first deals with the subcategory $\CHM^{\ab}(k)_{\Q}$ and then treats the whole $\DM^{\ab}(k)_{\Q}$. Both steps use the fact that homological and numerical equivalence coincide on abelian varieties in characteristic zero.

In positive characteristic homological and numerical equivalence are not known to coincide. The only known result is due to Clozel.

\begin{theorem}\cite{clozel}
Given an abelian variety over a finite field, the set of prime numbers $\ell$ for which numerical and $\ell$-adic homological equivalence coincide is of positive density.
\end{theorem}

Combining Wildeshaus' method with this result  one can show the following.

\begin{theorem}\label{cons intro}
Suppose the base field $k$ to be finite. Let $f: X \rightarrow Y$ be a morphism in $\DM^{\ab}(k)_{\Q}$. If $R_{\ell}(f)$ is an isomorphism for almost all primes $\ell$, then $f$ itself is an isomorphism.
\end{theorem}

Although this result is probably enough for applications over finite fields, it is intellectually unsatisfactory: for instance we cannot deduce, even for a single prime $\ell$, that the functor $R_{\ell}$ is conservative. To go further we need to restrict to Chow motives. 
\begin{theorem}\label{cons intro pure}
Let $k$ be a finite field. For any prime $\ell$ invertible in  $k$ the $\ell$-adic realization functor is conservative when restricted to $\CHM^{\ab}(k)_{\Q}$.
\end{theorem}

It is fun to notice how conservativity and the equality between homological and numerical equivalence are related also "in the other direction". For instance we show the following.

\begin{theorem}\label{cons intro2}
Let $k$ be a finite field and $\ell$ be a prime number invertible in  $k$. 
Suppose that, for all totally real number fields $F$ and all places $\lambda$ of $F$ above $\ell$, the $\lambda$-adic realization functor is conservative when restricted to $\DM^{\ab}(k)_{F}$. Then the $\ell$-adic homological equivalence  coincides with numerical equivalence on abelian varieties over $k$.
\end{theorem}

There are two tools in the proofs of these results. The first, valid over any field, is Kimura finiteness, which is a first approximation to conservativity. The other one is the classical fact, due to Tate, that abelian varieties over finite fields have sufficiently many complex multiplications. This allows to decompose their motives in very small direct factors. 

\subsection*{Organization of the paper} 
Section \S \ref{reminder} recalls results on motives of abelian type such as Kimura finiteness. In Section \S \ref{autoduality} we deduce the  main technical result (Proposition \ref{preliminar3}), inspired by Hodge Theory, which is valid over any field.
Section \S \ref{finitefields} recalls the theorem of Tate on endomorphisms of abelian varieties over finite fields and the results from \cite{clozel}. In Section  \S \ref{sectionhomnum} we will combine their results with Proposition \ref{preliminar3} and deduce Theorem \ref{cons intro pure}. Theorems \ref{cons intro} and  \ref{cons intro2} are explained in Section \ref{sectionmixed}.

\subsection*{Acknowledgments} I would like to thank Olivier Benoist, Fran\c{c}ois Charles, Fr\'ed\'eric D\'eglise,  Javier Fres\'an,  Peter Jossen, Marco Maculan and Charles Vial for useful comments.

\section{The motive of an abelian variety}\label{reminder}
We recall in this section classical results on motives of abelian type.
Let  $k$ be a base field, $F$ be a field of coefficients of characteristic zero and $\CHM(k)_{F}$ be the category of Chow motives over $k$ with coefficients in $F$ (for generalities, we refer to \cite{Andmot}).

\begin{theorem}\label{classic} Let $A$ be an abelian variety  of dimension $g$, $\End(A)$ its ring of endomorphisms (as an abelian variety) and $M(A)\in \CHM(k)_{F}$ its motive. Then the following holds.
\begin{enumerate}
\item\label{kunneth}  \cite{DeMu}  The motive $M(A)$ admits a K\"unneth decomposition 
\[M(A)=\bigoplus_{i=0}^{2g} \mathfrak{h}^i(A)\] natural in $\End(A)$. Moreover  $ \mathfrak{h}^0(A)$ is the unit object $\one$.
\item\label{kunn} \cite{Ku1} There is a canonical isomorphism \[\mathfrak{h}^i(A)=\Sym^i \mathfrak{h}^1(A).\]
\item\label{kings}   \cite[Proposition 2.2.1]{Ki}  The action of $\End(A)$ on $\mathfrak{h}^1(A)$ (coming from naturality in (\ref{kunneth}))  induces an isomorphism of algebras \[\End(A)\otimes_{\Z}F= \End_{\CHM^{\ab}(k)_{F}}(\mathfrak{h}^1(A))\]
and if $A$ is isogenous to $B\times C$, then $\mathfrak{h}^1(A) =  \mathfrak{h}^1(B) \oplus \mathfrak{h}^1(C) .$
\item\label{polarization}   \cite{Ku2} The classical isomorphism in $\ell$-adic cohomology  induced by a polarization $H_{\ell}^1(A)\cong H_{\ell}^1(A)^{\vee}(-1)$ lifts into an isomorphism
 \[\mathfrak{h}^1(A)\cong\mathfrak{h}^1(A)^{\vee}(-1).\] 
 \item\label{lefschetz}   \cite{Ku2} The Lefschetz decomposition of the $\ell$-adic cohomology $H_\ell^{2i}  (A)$ induced by a polarization, lifts into a decomposition of the motive $\mathfrak{h}^{2i}(A)$. 
\end{enumerate}
\begin{corollary}\label{ultimo} We keep notations from the theorem above. The following holds.
\begin{enumerate}
\item\label{lefsche} The motive  $\one(-1)$ is a direct factor of 
 $\mathfrak{h}^1(A)\otimes\mathfrak{h}^1(A).$ 

 \item\label{boh} A map $f:\mathfrak{h}^1(A)\rightarrow\mathfrak{h}^1(A)^{\vee}(-1)$ whose $\ell$-adic realization is zero must be zero.

\end{enumerate} 
\end{corollary}
\begin{proof}
Using the Lefschetz decomposition of Theorem \ref{classic}(\ref{lefschetz}) we have that $\one(-1)$ is a direct factor of $\mathfrak{h}^2(A)$ (recall that, by Theorem \ref{classic}(\ref{kunneth}) , we have $ \mathfrak{h}^0(A)=\one$). On the other hand $\mathfrak{h}^2(A)$ is a direct factor of 
 $\mathfrak{h}^1(A)\otimes\mathfrak{h}^1(A)$ by Theorem \ref{classic}(\ref{kunn}), this shows (\ref{lefsche}).
 
 To show  (\ref{boh}), we compose $f$  with an isomorphism $\mathfrak{h}^1(A)\cong\mathfrak{h}^1(A)^{\vee}(-1)$ of Theorem \ref{classic}(\ref{polarization}). This reduces to show that the realization is injective on 
 $\End_{\CHM^{\ab}(k)_{F}}(\mathfrak{h}^1(A))$, which is clear by Theorem \ref{classic}(\ref{kings}).
 \end{proof}
\end{theorem}
\begin{definition}\label{pure abelian}
Define $\CHM^{\ab}(k)_{F}$ to be the smallest rigid and pseudo-abelian full subcategory of  $\CHM(k)_{F}$ containing motives of abelian varieties. A motive in $\CHM^{\ab}(k)_{F}$  is called "of abelian type".

A motive $X$ of abelian type is pure, if there is a realization functor $R$ such that the cohomology groups of $R(M)$ are all zero except in one degree. In this case the degree will be called the weight of $X$. Moreover such an $X$ is said to be of dimension $d$, if the only non zero cohomology group of  $R(M)$ is of dimension $d$. In this case we define $\det X$ as $\wedge^d X$ if the weight is even and as  $\Sym^d X$ if the weight is odd. Similarly we define $\det f$ for a morphism $f: X \rightarrow Y$ between pure motives of same degree and dimension.
\end{definition}
\begin{remark}
The notions above do not actually depend on the choice of the realization functor $R$ \cite[Lemme B.1.4]{AK}. Note also that an odd object $X$ of dimension $d$ in our sense, is of dimension $-d$ in Kimura's sense.\end{remark}

\begin{theorem}\label{kimura} Let $X$ be a motive of abelian type and $R$ be a realization functor with respect to a fixed Weil cohomology. Then the following holds.
\begin{enumerate}
\item\label{phantom} \cite[Corollary 7.3]{kim} If $R(X)$ is zero then $X$ itself is zero.
\item\label{invertible} \cite[Corollaire 3.19]{AndBour} If $X$ is of dimension one, then 
$X \otimes X^{\vee}\cong \one.$
\item\label{small} \cite[Corollary 3.7]{jannsen} If $X$ is of dimension one, then 
\[\End_{\CHM^{\ab}(k)_{F}}(X) =F \cdot \id.\]
\item\label{determinant}  \cite[Lemma 3.2]{OS}  If $R(X)$ is concentrated in even degree (respectively odd), and of total dimension $d$, then $X^\vee= \wedge^{d-1} X  \otimes (\det X )^\vee$ \newline (respectively  $X^\vee= \Sym^{d-1} X  \otimes (\det X )^\vee$).
\item\label{projectors} \cite[Corollary 7.8]{kim} Any decomposition of $X$ as homological motive (with respect to $R$), or as numerical motive, lifts to a decomposition of $X$ in $\CHM^{\ab}(k)_{F}$.
\item\label{endomorphism}  \cite[Corollary 7.9]{kim}  Let $f:X \rightarrow X$ be an endomorphism. If $R(f)$ is an isomorphism then $f$ is an isomorphism too.
\item\label{lifting} \cite[Corollaire 3.16]{AndBour} Let $Y$ be another motive of abelian type. If $X$ and $Y$ are isomorphic as  homological motives (or numerical motives) then they are isomorphic in $\CHM^{\ab}(k)_{F}$.
\end{enumerate}
\end{theorem}

\begin{corollary}\label{corollaryweight}
Any motive of abelian type can be written as a sum of pure motives. Any pure motive of weight $n$ can be written as  a direct factor of $\mathfrak{h}^1(A)^{\otimes n+2m}(m)$, for some abelian variety $A$ and some integer $m$.
\end{corollary}

\begin{proof}
By K\"unneth formula, we have that \[M(A)(m)\otimes M(A')(m')=M(A\times A')(m+m'),\] hence any motive of abelian type is a direct factor of a finite sum of the form $\oplus_i M(A_i)(m_i)$. Write the K\"unneth decompositions of the motives  $M(A_i)$ (Theorem \ref{classic}(\ref{kunneth})). They induce a  K\"unneth decomposition for the homological motive associated with $X$. Using Theorem \ref{kimura}(\ref{projectors}) we lift this into a decomposition of $X$ refining the K\"unneth decomposition of $\oplus_i M(A_i)(m_i)$. This shows the first part of the statement and moreover  that $X_n$, the pure factor of $X$ of weight $n$,  is a direct factor of $\oplus_i \mathfrak{h}^{n+2m_i}(A_i)(m_i)$. 
 
 Now, by Theorem \ref{classic}(\ref{kunn}),  the motive $\oplus_i \mathfrak{h}^{n+2m_i}(A_i)(m_i)$ is a direct factor of  $\oplus_i \mathfrak{h}^1(A_i)^{\otimes n+2m_i}(m_i)$. Take a positive integer $m$ bigger than all the $m_i$ and use Corollary \ref{ultimo}(\ref{lefsche}) to deduce that  $\oplus_i \mathfrak{h}^1(A_i)^{\otimes n+2m_i}(m_i)$ is a direct factor of  $\oplus_i \mathfrak{h}^1(A_i)^{\otimes n+2m}(m)$.
 
 On the other hand $\mathfrak{h}^1(\times_i A_i) = \oplus_i \mathfrak{h}^1(A_i)$ by Theorem \ref{classic}(\ref{kings}), hence the motive $\oplus_i \mathfrak{h}^1(A_i)^{\otimes n+2m}(m)$ is a direct factor of $\mathfrak{h}^1(\times_i A_i)^{\otimes n+2m}(m)$. Putting all together we deduce that $X_n$ is a direct factor of $\mathfrak{h}^1(\times_i A_i)^{\otimes n+2m}(m)$. 
\end{proof}

\section{Autoduality of motives}\label{autoduality}

We keep the notations from the previous section. In this section we prove a criterion to check conservativity of realization on Chow motives of abelian type.

\begin{proposition}\label{preliminar1}
Let $X$ and $Y$ be two motives of abelian type and ${f:X \rightarrow Y}$  and $g:Y \rightarrow X$ be two morphisms. Let  $R$ be a realization functor and suppose that $R(f)$ and $R(g)$ are isomorphisms. Then $f$ and $g$ are isomorphisms too. \end{proposition} 

\begin{proof}
We do the proof for $f$ (of course the situation is symmetric). The realization of $g\circ f$ is an isomorphism, so, by Theorem \ref{kimura} (\ref{endomorphism}), $g\circ f$ is an isomorphism too. In particular, we can find a morphism $h: X \rightarrow X$ such that $(h\circ g) \circ f= \id_X$. This implies that $ f \circ (h\circ g): Y \rightarrow Y$ is a projector defining $X$ as a direct factor of $Y$, hence $Y= X\oplus H$. But the factor $H$ has zero realization, so it is actually zero, which means that $f$ and $(h\circ g) $ are one the inverse of the other.
\end{proof}

\begin{proposition}\label{preliminar2}
Let $X$ and $Y$ be two pure motives of abelian type of same weight and dimension. Let $f:X \rightarrow Y$  be a morphism such that $\det f$ is an isomorphism. Then $f$ is an isomorphism too.
\end{proposition} 

\begin{proof}
We call $n$ the weight and $d$ the dimension and write the proof for $n$ even (the odd case is analogous).
Let us fix a realization functor $R$. As $\det f$ is an isomorphism, then $R(f)$ must be an isomorphism. This implies that  $R(\wedge^i f )$ is an isomorphism, for any $i$. Then the realization of the map  
\[(\wedge^{d-1} f)^\vee \otimes h : (\wedge^{d-1} Y)^\vee \otimes \det Y \rightarrow (\wedge^{d-1} X)^\vee \otimes \det X\]
 is an isomorphism.  Using Theorem \ref{kimura}(\ref{determinant}) we have constructed a map
 ${g: Y \rightarrow X}$ whose realization is an isomorphism. We conclude using Proposition \ref{preliminar1}.
\end{proof}

\begin{proposition}\label{preliminar3}
Suppose that for all pure motives $X \in \CHM^{\ab}(k)_{F}$ of even weight $n$ and dimension one we have an isomorphism
\[X \cong X^\vee(-n).\]
Then any realization functor is conservative.
\end{proposition}
\begin{proof}
Let us fix a realization functor $R$ and $f:X\rightarrow Y$ a map of abelian motives such that $R(f)$ is an isomorphism. The aim is to show that $f$ is also an isomorphism.

First, write two (finite) decompositions   $X=\oplus_n X_n$ and $Y=\oplus_n Y_n$, where $X_n$ and $Y_n$ are pure of weight $n$ (Corollary \ref{corollaryweight}). The map $f$ induces morphisms $f_n:X_n \rightarrow Y_n$ (but in general $f$ is not just the sum of the $f_n$). Note that $R(f_n)$ is an isomorphism. It is enough to show that each $f_n$ is an isomorphism.  Indeed,  the inverses $g_n$ of the $f_n$ induce a morphism $g:Y \rightarrow X$ allowing to apply Proposition \ref{preliminar1}.

We reduced to the case where $X$ and $Y$ are pure. By Proposition \ref{preliminar2} it is enough to show that $\det f$ is an isomorphism, hence we reduced to the case where $X$ and $Y$ are pure of dimension one. 

By Proposition \ref{preliminar1}, it is enough  to construct a morphism $g:Y \rightarrow X$ whose realization is an isomorphism (or equivalently non zero). It is constructed as follows
\[Y =Y\otimes \one \cong (Y\otimes X^\vee) \otimes X \cong (X \otimes Y^\vee)\otimes X \rightarrow  (X \otimes Y^\vee)\otimes Y\ = X\]
where the first isomorphism comes from Theorem \ref{kimura}(\ref{invertible}) and the second from the hypothesis.
\end{proof}

\section{Abelian varieties over finite fields}\label{finitefields}
We recall here some classical results on abelian varieties over finite fields due to Tate \textit{et al.} and we give some consequences. In all the section we fix a polarized abelian variety $A$ of dimension $g$ over a finite field $k$. We denote by $\End(A)$ the ring of endomorphism of $A$,  we write $\End^0(A)$ for $\End(A)\otimes_\Z \Q$ and $*$ for the Rosati involution on it.

\begin{theorem}\label{classical}
With the above notations the following holds.
\begin{enumerate}
\item\label{tate} \cite{Tate}  A maximal commutative $\Q$-subalgebra of $\End^0(A)$ has dimension $2g$.
\item\label{Yu} \cite[\S 2.2]{yuchia} There exists a maximal commutative $\Q$-subalgebra $B$ of $\End^0(A)$ which is $*$-stable.
\item\label{CM}\label{mumford} \cite[pp. 211-212]{MumfordTata} The algebra $B$ is a finite product of CM number fields ${B=L_1\times \cdots\times L_t}$ and $*$ acts as the complex conjugation on each factor.
\item\label{compositum}  \cite[Proposition 5.12]{goro} The compositum of the number fields $L_1,\ldots, L_t$ is a CM field. There exist a CM number field $L$, which is Galois over $\Q$ and which contains the compositum.
\end{enumerate}
\end{theorem}
Write $\Sigma_i$ for the set of embeddings of $L_i$ in $L$ and   $\Sigma$ for the disjoint union of the $\Sigma_i$ (with $i$ varying). Write $\bar{\cdot}$ for the action on $\Sigma$ induced by composition with the complex conjugation. 
\begin{corollary}\label{decomposition}
We keep the notations as above, in particular $L$ is defined in Theorem \ref{classical}(\ref{compositum}). In $\CHM^{\ab}(k)_{L}$ the motive $\mathfrak{h}^1(A)$ decomposes into a sum of $2g$ motives of dimension one
\[\mathfrak{h}^1(A) =\bigoplus_{\sigma \in \Sigma} M_\sigma,\]
where the action of $b\in L_i$ on $M_\sigma$ induced by Theorem \ref{classic}(\ref{kings}) is given by multiplication by $\sigma(b)$ if $\sigma \in \Sigma_i$ and by multiplication by zero otherwise. 

Moreover the isomorphism $p:  \mathfrak{h}^1(A)\cong\mathfrak{h}^1(A)^{\vee}(-1)$ of Theorem \ref{classic}(\ref{polarization}) restricts to an isomorphism \[M_\sigma\cong M_{\bar{\sigma}}^\vee(-1)\] for all $\sigma$, and to the zero map 
\[ M_\sigma \stackrel{0}{\longrightarrow} M_{\sigma'}^\vee(-1) \]
for all $\sigma'\neq {\bar{\sigma}}$.
\end{corollary}
\begin{proof}
Consider the injection $L_1\times \cdots\times L_t  \hookrightarrow  \End^0(A).$ By Theorem \ref{classic}(\ref{kings}), we deduce an injection $(\prod_i L_i)\otimes L  \hookrightarrow  \End_{\CHM^{\ab}(k)_{L}}(\mathfrak{h}^1(A)).$ Each projector of $(\prod_i L_i)\otimes L \cong \prod_i L^{[L_i:\Q]}$ defines a factor $M_{\sigma}.$

The last part of the statement can be checked after realization because of Corollary \ref{ultimo}(\ref{boh}). It is then a consequence of Theorem \ref{classical}(\ref{mumford}).
\end{proof}

\begin{definition}\label{clozelprimes}
We keep notations from the theorem above and define $L_0$ to be $L\cap \R$.

Following Clozel we define a set of prime numbers $\Clo(A,*,B)$ as those primes $\ell$ (different from the characteristic of $k$), such that there is a place $\lambda$ of $L_0$ above $\ell$ such that the $\lambda$-adic completion of $L_0$  does not contain $L$.

If there are several $B\subset \End^0(A)$ as in the theorem above we can let $B$ vary and consider the union of the  $\Clo(A,*,B)$. We will call it $\Clo(A,*)$ or  simply $\Clo(A)$.
\end{definition}
\begin{proposition}\label{propclozel}\cite[\S 3]{clozel} 
Given a totally real number field $F$ and an imaginary quadratic extension $F'$, the set of primes $\ell$ such that there is a place $\lambda$ of $F$ above $\ell$ such that the $\lambda$-adic completion of $F$  does not contain $F'$ is of positive density.

In particular, the set $\Clo(A,*,B)$ as subset of the set of prime numbers is of positive density.
\end{proposition}
\begin{theorem}\label{clozel}\cite{clozel} 
Let $\ell$ be a prime number in $\Clo(A)$, then numerical and $\ell$-adic homological equivalence on $A$ (and all powers of $A$) coincide.
\end{theorem}
The improvement on powers of $A$ is due to Milne \cite[Proposition B.2]{Milnenew}.

\section{Conservativity on Chow motives}\label{sectionhomnum}
In all this section the base field $k$ is finite. We show here Theorem \ref{cons intro pure} from the Introduction.
\begin{theorem}\label{main}
Suppose that  the base field $k$ is finite and that the field of coefficients $F$ verifies that $F\cap \overline{\Q}$ is totally real. Then for any \[X \in \CHM^{\ab}(k)_{F}\]of even weight $n$ and dimension one we have an isomorphism
\[X \cong X^\vee(-n).\] 
\end{theorem}
\begin{proof} Let us start with some reduction steps.
First, note that  it is enough to have such an isomorphism  in the category of numerical motives (by Theorem \ref{kimura}(\ref{lifting})).  Recall that the category of numerical motives is semisimple \cite{Jann}.

We claim that  the isomorphism class of the numerical motive $X$ exists with coefficients in $F$ if and only if it exists with coefficients in $F\cap  \overline{\Q}$.  To show this claim it is enough to show that there are no more simple objects with coefficients in  $F $ then with coefficients in $F\cap \overline{\Q}$.  As the  endomorphisms algebra of a simple object  is a division algebra, it is enough to show that if $D$ is a division algebra on $F\cap  \overline{\Q}$, then also $D\otimes_{F\cap  \overline{\Q}} F$ is a division algebra. This is certainly classical, but we do not know a reference. It can be deduced for exemple from \cite[Th\'eor\`eme 6.1]{GroBra}.

The claim reduces  the question whether $X$ and  $X^\vee(-n)$ are isomorphic  to the case $F \subset \overline{\Q}$. In this case such an $X$ is actually already defined with coefficients in a number field. Hence we can work with the case where $F$ is a totally real number field.

\

Consider two totally real number fields $F \subset  L$. We claim that the statement for $L$ implies the statement for $F$. To show this claim we work again with numerical motives. Let $X$ be a motive as in the statement, with coefficients in $F$. Note that $\Hom(X, X^\vee(-n))$ and $\Hom(X^\vee(-n),X)$ are at most one dimensional. Moreover, passing to coefficients in $L$ corresponds to apply $\otimes_F L$ to these $\Hom$. 
Hence, if the relation $f\circ g = \id$ can be satisfied with coefficients in $L$ then it can be satisfied also with coefficients in $F$.

\

We can now show the statement. We are reduced to the case where $F$ is a totally real number field as big as we want. Any motive $X$ as in the statement can be written as  a direct factor of $\mathfrak{h}^1(A)^{\otimes n+2m}(m)$, for some abelian variety $A$ and some integer $m$, by Corollary \ref{corollaryweight}. After twist, we can suppose that $X$ is a direct factor of $\mathfrak{h}^1(A)^{\otimes n}$, with $n$ even.

Consider now $L$ as defined in Theorem \ref{classical}(\ref{compositum}). In $\CHM^{\ab}(k)_{L}$ the motive $\mathfrak{h}^1(A)$ decomposes into a sum of  motives of dimension one
\[\mathfrak{h}^1(A) =\bigoplus_{\sigma \in \Sigma} M_\sigma, \]
as explained in Corollary \ref{decomposition}.

We can suppose that $F$ contains the biggest totally real number field in $L$.  In particular we can decompose the motive 
$\mathfrak{h}^1(A)^{\otimes n}$ in $\CHM^{\ab}(k)_{F}$ into a sum of motives of dimension two of the form

\[ (M_{\sigma_1}\otimes \cdots \otimes M_{\sigma_n}) \oplus  (M_{\overline{\sigma_1}}\otimes \cdots \otimes M_{\overline{\sigma_n}}), \]
where $\sigma_i \in  \Sigma$ and $\overline{\cdot}$ is the action induced by complex conjugation.

Again we can work with numerical motives. By semisimplicity, the isomorphism class of $X$ appear in a motive of the form 
\[ Y = (M_{\sigma_1}\otimes \cdots \otimes M_{\sigma_n}) \oplus  (M_{\overline{\sigma_1}}\otimes \cdots \otimes M_{\overline{\sigma_n}}) ,\] hence we can then suppose that $X$ is direct factor of $Y$. Moreover, we can see $X$ as a direct factor of $Y$ also in the category of Chow motives because of Theorem \ref{kimura}(\ref{projectors})

By Corollary \ref{decomposition}, the morphism $p^{\otimes n}$ induces an isomorphism between  $Y$ and $Y^\vee(-n)$. It suffices to show that the restriction of this isomorphism to $X$ induces an isomorphism between $X$ and  $X^\vee(-n).$ This can be checked after realization by Proposition \ref{preliminar1}. In an equivalent way, the realization can be seen as a pairing on $R(Y)$ and we have to check that $R(X)$ is not an isotropic line. The pairing is perfect and symmetric on $R(Y)$ so at most two lines are isotropic. By Corollary \ref{decomposition}, $R(M_{\sigma_1}\otimes \cdots \otimes M_{\sigma_n}) $ and $R (M_{\overline{\sigma_1}}\otimes \cdots \otimes M_{\overline{\sigma_n}}) $ are isotropic lines, so we have to check that $R(X)$ is not one of these two lines.

We can choose $R$ to be the $\lambda$-adic realization, with $\lambda$ one of the primes of $F$ as in the Proposition \ref{propclozel} (to be applied to $F'$ the compositum of $F$ and $L$). In this way the complex conjugation acts on the coefficients sending $R(M_{\sigma_1}\otimes \cdots \otimes M_{\sigma_n}) $ to $R (M_{\overline{\sigma_1}}\otimes \cdots \otimes M_{\overline{\sigma_n}}) $ and fixing $R(X)$. This implies that they are not the same lines and concludes the proof.
 \end{proof}

\begin{corollary}\label{corconservativity}
Suppose that the base field $k$ is finite and that the field of coefficients $F$ verifies that $F\cap \overline{\Q}$ is totally real. Then any realization functor is conservative on $\CHM^{\ab}(k)_{F}$.
\end{corollary}
\begin{proof}
Combine  the previous theorem with Proposition \ref{preliminar3}.
\end{proof}

 \begin{remark}\label{comments}
The condition on $F$ is a necessary hypothesis in the theorem. Indeed, if $E$ is an elliptic curve with CM multiplication by a field $L$, then one can easily check that $\mathfrak{h}^1(A) \in \CHM^{\ab}(k)_{L}$ decomposes as $X\oplus Y$ with $X \cong Y^\vee(-1)$, but in general  $X^{\otimes 2}$ and $Y^{\otimes 2}$ are not  isomorphic (apply the realizations).

Instead, the corollary on conservativity should hold without the extra assumption on the field of coefficients, but we are not able to show it. Note that this would have deep consequences such as the fact that homological and numerical equivalence coincide on abelian varieties over finite fields for all primes $\ell$.  Indeed, by Corollary \ref{decomposition} (with Corollary \ref{corollaryweight}) any simple motive of abelian type with coefficients in $\overline{\Q}$ is of dimension one. Then any map between such motives whose realization is  nonzero would have an inverse, hence it would also be numerically nonzero.
\end{remark}

\section{Conservativity on mixed motives}\label{sectionmixed}
In all this section the base field $k$ is finite. We study the conservativity of the realization functors on the category $\DM^{\ab}(k)_{\Q}$ (Definition \ref{defabeliancat}). The results are weaker then the previous section.

\begin{theorem}\label{mixedconservativity}
Let $k$ be a field of positive characteristic and $f: X \rightarrow Y$ be a morphism in $\DM^{\ab}(k)_{\Q}$. If $R_{\ell}(f)$ is an isomorphism for almost all primes $\ell$, then $f$ itself is an isomorphism.
\end{theorem}
\begin{proof}
 First note \cite[Remark 5.6]{Anc15} that  $\DM^{\ab}(k)_{\Q}$ is the smallest triangulated category containing Chow motives of abelian type. 
 Now, $\DM^{\ab}(k)_{\Q}$ has a canonical weight structure (in the sens of \cite[\S 6]{Bon}), whose heart is $\CHM^{\ab}(k)_{\Q}$ \cite[Proposition 1.2 and its proof]{WildPic}. Moreover, this weight structure is finite, hence only finitely many abelian varieties are needed to generate $X$ and $Y$. Let $A$ be the product of those, $\ell$ be a prime for which numerical and $\ell$-adic homological equivalence coincide on powers of $A$, and $\mathcal{C}$ be the smallest triangulated, rigid and pseudoabelian category containing the motive of $A$. Note that $X,Y \in \mathcal{C}$. We can now apply Wildeshaus's methods \cite[proofs of 1.10-1.12]{WildPic} to $\mathcal{C}$, to conclude that the $\ell$-adic realization (for the $\ell$ we chose) is conservative on  $\mathcal{C}$.
\end{proof}
\begin{theorem}
Let $k$ be a finite field and $\ell$ be a prime number invertible in  $k$. 
Suppose that, for all totally real number fields $F$ and all places $\lambda$ of $F$ above $\ell$, the $\lambda$-adic realization functor is conservative when restricted to $\DM^{\ab}(k)_{F}$. Then the $\ell$-adic homological equivalence  coincides with numerical equivalence on abelian varieties over $k$.
\end{theorem}
\begin{proof}
Suppose that there is an algebraic cycle $Z$ of codimension $i$ on an abelian variety $A$ 
which is numerically trivial but with $\ell$-adic class non trivial. 
We look at it as an element in $\Hom_{\CHM^{\ab}(k)_{\Q}}(M(A),\mathbb{L}^{\otimes i})$, where $\mathbb{L}$ is the Lefschetz motive. Using Theorem \ref{classic}(\ref{kunn}) we can look at it as an element $\alpha \in \Hom_{\CHM^{\ab}(k)_{\Q}}(\mathfrak{h}^1(A)^{\otimes 2i},\mathbb{L}^{\otimes i})$

Consider now $L$ as defined in Theorem \ref{classical}(\ref{compositum}) and let be $F$ the biggest totally real number field in it. In $\CHM^{\ab}(k)_{L}$ the motive $\mathfrak{h}^1(A)$ decomposes into a sum of  motives of dimension one
\[\mathfrak{h}^1(A) =\bigoplus_{\sigma \in \Sigma} M_\sigma, \]
as explained in Corollary \ref{decomposition}.
 In particular we can decompose the motive 
$\mathfrak{h}^1(A)^{\otimes 2i}$ in $\CHM^{\ab}(k)_{F}$ into a sum of motives of dimension two of the form

\[ (M_{\sigma_1}\otimes \cdots \otimes M_{\sigma_{2i}}) \oplus  (M_{\overline{\sigma_1}}\otimes \cdots \otimes M_{\overline{\sigma_{2i}}}), \]
where $\sigma_i \in  \Sigma$ and $\overline{\cdot}$ is the action induced by complex conjugation. This induces a decomposition of the morphism $\alpha$ and there exists one of its components
\[ f \in \Hom_{\CHM^{\ab}(k)_{F}}(X,\mathbb{L}^{\otimes i})\]
which is numerically trivial but whose realization is non zero, with \[X=(M_{\sigma_1}\otimes \cdots \otimes M_{\sigma_{2i}}) \oplus  (M_{\overline{\sigma_1}}\otimes \cdots \otimes M_{\overline{\sigma_{2i}}})\] for a certain choice of the $\sigma_1,\ldots,\sigma_{2i}$.

Consider the isomorphism $p$ as in Corollary \ref{decomposition} and define
\[g=f^{\vee}(-2i)\circ p^{\otimes 2i}  \in \Hom_{\CHM^{\ab}(k)_{F}}(\mathbb{L}^{\otimes i},X).\]
As $f$ is numerically trivial, we must have $g \circ f = 0.$

\

On the other hand, in the category\footnote{We take the embedding of $\CHM^{\ab}(k)_{F}$ into $\DM^{\ab}(k)_{F}$ to be covariant.} $\DM^{\ab}(k)_{F}$, we can complete $f$ into a triangle
\[ C \longrightarrow X \stackrel{f}{\longrightarrow} \mathbb{L}^{\otimes i} \]
and $g$ must factorise into a morphism 
\[h: \mathbb{L}^{\otimes i}  \longrightarrow C.\]

As the realization of $f$ is non zero, the realization of $h$ is a non zero map between vector spaces of dimension one, hence it is an isomorphism. Conservativity implies that $h$ is an isomorphism too, hence $C\cong \mathbb{L}^{\otimes i}$. This means that the triangle above is a triangle between Chow motives.
By \cite[Corollary 4.2.6]{TMF}, the triangle splits, hence $X \cong \mathbb{L}^{\otimes i} \oplus \mathbb{L}^{\otimes i}$. In particular numerical and homological equivalence coincide on $\Hom_{\CHM^{\ab}(k)_{F}}(X,\mathbb{L}^{\otimes i})$, which gives a contradiction.
\end{proof}

\bibliographystyle{alpha}	
\bibliography{masterbib}
\end{document}